\documentclass[11pt]{amsart}
\usepackage{latexsym}
\usepackage{amsmath,amsthm,amssymb,amsfonts}
\allowdisplaybreaks
\newtheorem{theorem}{Theorem}
\newtheorem{remark}{Remark}
\newtheorem{lemma}[theorem]{Lemma}

\newtheorem{definition}[theorem]{Definition}
\newtheorem{example}{Example}

\begin{document}
\vspace{0.5cm}
%
%
%
%

\title{Steady-State Solutions in an Algebra of Generalized Functions:\\
 \Small{\rm Lightning, Lightning Rods and Superconductivity}}
\author{Todor D. Todorov} 
 \address{Mathematics Department\\                
                        California Polytechnic State University\\
                        San Luis Obispo, California 93407, USA.}									
\email{ttodorov@calpoly.edu}
\thanks{*For a glimpse of Academician Christo Ya. Christov mathematics and physics heritage we refer to Mathematics  Genealogy  Project  (http://www.genealogy.math.ndsu.nodak.edu/id.php?id=177888).  The author of this article is particularly grateful to Professor Christov for introducing him to the theory of generalized functions - through infinitesimals - in an era when infinitesimals had already been expelled from mathematics and almost forgotten}.
\date{}
\maketitle

Dedicated to the memory of Christo Ya. Christov* on the 100-th anniversary of his birth and to James Vickers on the occasion his 60-th birthday.
\\
%
%
\begin{center} 

\textbf{Abstract}
\end{center}

    Formulas for the solutions of initial value problems for ordinary differential equations with singular $\delta^{(n)}$-like driving terms are derived in the framework of an algebra of generalized functions (of Colombeau type) over a field of generalized scalars. Some of the solutions might have physical meaning - such as of the electrical current after lightning or under superconductivity - but do not have counterparts in the theory of Schwartz distributions. What is somewhat unusual (compared with other similar works) is the involvement of infinitely large constants, such as $\delta(0)$, in some of the formulas for the solutions. 

\section{\bf Introduction and Notation}\label{S: Introduction and Notation}
\begin{enumerate}

\item[1.] Our notations are similar to those in Bremermann~\cite{hBremermann65} and Vladimirov~\cite{vVladimirov}: $\mathbb N, \mathbb N_0, \mathbb R$, $\mathbb R_+$ and $\mathbb C$ denote the set of the natural, whole, real, positive real and complex numbers, respectively. $\mathcal E(\mathbb R)=\mathcal C^\infty(\mathbb R)$ stands for the class of $\mathcal C^\infty$-functions from $\mathbb R$ to $\mathbb C$ and $\mathcal D(\mathbb R)$ denotes the set of test-functions on $\mathbb R$. We denote by $\mathcal D^\prime(\mathbb R)$ the space of Schwartz distributions on $\mathbb R$ and by $\mathcal D_+^\prime(\mathbb R)$ the space of all distributions in $\mathcal D^\prime(\mathbb R)$ supported by the interval $[0, \infty)$. 

\item[2.] The main purpose of this article is to derive formulas for the solutions of the initial value problem: 
\begin{equation}\label{E: Steady-State}
P(d/dt)y=f(t), \quad y(0)=y^\prime(0)=\dots=y^{(k-1)}(0)=0,
\end{equation}
in an algebra $\widehat{\mathcal E(\mathbb R)}$ of generalized functions over a field of generalized scalars $\widehat{\mathbb C}$ (of Colombeau type) constructed in Todorov \& Vernaeve~\cite{TodVern08}.  Here $P\in\mathbb C[x]$ (and more generally, $P\in\widehat{\mathbb C}[x]$) is a polynomial in one variable of degree $k$ and $P(d/dt)$ stands for the corresponding differential operator with constant coefficients. The right hand side of the equation, $f$, is a distribution (and more generally, a generalized function in  $\widehat{\mathcal E(\mathbb R)}$). The solution of the intitial value problem (\ref{E: Steady-State}) are often called  \emph{steady-state solution} of the differential equation $P(d/dt)y=f(t)$, which explains the title of the article. We are mostly interested in initial value problems (\ref{E: Steady-State}) which do not admit distributional solutions even when $P\in\mathbb C[x]$ and $f$ is a distribution with compact support. 

\item[3.] Recall that if $T\in\mathcal D^\prime(\mathbb R)$ is a distribution and $x_0$ is a real number, then the value $T(x_0)\in\mathbb C$ exists in the sense of Robinson~(\cite{aRob66}, \S 5.3) if and only if $T$ is a continuous function on a neighborhood of $x_0$ (Vernaeve \& Vindas~\cite{hVernave2012}). Consequently, (\ref{E: Steady-State}) admits a (unique) distributional solution $y$ if and only if the equation $P(d/dt)y=f(t)$ admits a (particular) solution $y_p(t)\in\mathcal C^{(k-1)}(-\varepsilon, \varepsilon)$ for some $\varepsilon\in\mathbb R_+$. Here is a typical example of differential equation which does not admit a steady-state distributional solution.

\begin{example} The IVP 
$y^{\prime\prime}+\omega^2 y=\delta^\prime(t),\; y(0)=y^\prime(0)=0$, where $\omega\in\mathbb R_+$, does not admit a distributional solution. Indeed,  the general solution of the equation $y^{\prime\prime}+\omega^2 y=\delta^\prime(t)$ is presented by the family $y=c_1\cos{\omega t}+c_2\sin{\omega t}+ G_0^\prime(t)$, $c_1, c_2\in\mathbb C$, where $G_0(t)=\frac{1}{\omega}H(t)\sin{\omega t}$ is the Green function of the operator $\frac{d^2}{dt^2}+\omega^2$. Here $H\in\mathcal D^\prime(\mathbb R)$ stands for the Heaviside distribution  with a kernel, $h(t)$, the unit step function.We also have $G_0^\prime(t)=H(t)\cos{\omega t}$ and $G_0^{\prime\prime}(t)=\delta(t)-\omega^2 G_0(t)$. The initial condition $y(0)=0$ implies $c_1= -G_0^\prime(0)$ and $y^\prime(0)=0$ implies $c_2= -\frac{1}{\omega}[\delta(0)-\omega^2 G_0(0)]=-\frac{\delta(0)}{\omega}$ since $G_0(0)=0$. Thus the only candidate for a solution of our IVP is $y=-G_0^\prime(0)\cos{\omega t}-\frac{1}{\omega}\delta(0)\sin{\omega t}+ H(t)\cos{\omega t}$. However, neither $\delta(0)$, nor $G_0^\prime(0)$ exists in the theory of distributions. 
\end{example}
\item[4.] Strangely enough, it seems the above phenomenon - non-existence of a distributional solution - is widely unknown (if known at all); way too often initial value problems for steady state solutions without distributional solutions (as those  presented above) appear in the mathematical literature.  We have selected three typical examples (among many): (a) In (Edwards \& Penney~\cite{EdwardsPenney}, 7.6 Problems) the problems \# 1, 2, 4, 8 do not admit distributional solutions and are supplied with wrong answers. In (Schiff~\cite{jlSchiff}, p. 104) the function $g(t)=\sinh{t}$ is presented as the solution of $g^{\prime\prime}-g=\delta(t), g(0)=g^\prime(0)=0$ (although $g^\prime(0)=1$) and in the recent article (Strang~\cite{gStrang}, p. 1245) the function $g(t)=e^{at}$ is presented as the solution of $g^\prime-ag=\delta(t),\; g(0)=0$. The authors obviously are unaware that these initial value problems do not admit distributional solutions (and never  bothered to verify the correctness of the final result).

\item[5.] It is very tempting to try to solve the initial value problem (\ref{E: Steady-State}) by the Laplace transform method and this is what the authors cited above have tried to do. We should mention that: (a) The Laplace transform  will always produce the correct solution to  (\ref{E: Steady-State}) if  (\ref{E: Steady-State}) has a distributional solution. (b) If (\ref{E: Steady-State}) does not admit a distributional solution, the Laplace transform method (if applied correctly) will produce a particular solution of the differential equation which does not satisfy all initial conditions. However, the Laplace transform will not alarm us for non-existence of solutions (verifying the final result is always advisable).  (c) Strictly speaking the formulas in the usual tables of Laplace transforms (which appear everywhere in the textbooks and also in Wikipedia) are logically inconsistent (Todorov~\cite{tdTodorovAxioms11}). For example, these tables suggest wrongly that $\mathcal L^{-1}\big(\frac{1}{s^2+1}\big)=\sin{t}$ instead of the correct one $\mathcal L^{-1}\big(\frac{1}{s^2+1}\big)=H(t)\sin{t}$. For the readers who prefer to use the Laplace transform (over the Fourier transform) for the purpose of finding a particular solution of a differential equation, we recommend the table ''LaplaceTable'' (which incudes $\mathcal L^{-1}$ as well) which can be found in the webpage of the author of this article. Alternatively, the reader might refer to to the inversion integral formula for $\mathcal L^{-1}$ (Folland~\cite{gbFolland}), \S 8.2).

\item[6.] Notice that the above formula $y=-G_0^\prime(0)\cos{\omega t}-\frac{\delta(0)}{\omega}\sin{\omega t}+ H(t)\cos{\omega t}$ (although meaningless within the Schwartz theory of distributions) make perfect sense in any of the numerous algebras of generalized functions of Colombeau type (Colombeau~\cite{jfCol90}). This is because the value $T(0)$ is well defined - as a generalized number - for any distribution $T$. The main goal of this article is to show that these formulas (and many similar to them) not only make sense, but they are actually the solution (steady-state solution) of the initial value problem from which these formulas originate. 

\item[7.] Except the example at the end of our article, the solutions of (\ref{E: Steady-State}) in this article are linear combinations of Schwartz distributions with coefficients in $\widehat{\mathbb R}$. The fact that the set of scalars $\widehat{\mathbb C}$ of the algebra of generalized functions  $\widehat{\mathcal E(\mathbb R)}$  (Todorov \& Vernaeve~\cite{TodVern08}) is a field (an algebraically closed non-Archimedean field) - not a ring with zero divisors as in Colombeau~\cite{jfCol90}) - is important for our approach; it allows us to transfer the familiar linear algebra arguments from the classical theory of ordinary linear differential equations with constant coefficients to the similar equations in $\widehat{\mathcal E(\mathbb R)}$. 

\item[8.]  In the last section of the article we endow some of our initial value problems with physical interpretation inherited from the \emph{Kirchoff law}: $LI^{\prime\prime}+R\, I^\prime+\frac{1}{C}\,I= V^\prime(t),\quad I(0)=0,\quad  LI^\prime(0)=0$, for an electrical $LRC$-circuit with inductance $L$, resistance $R$, capacitance $C$, driving electromotive force $V(t)$ and (steady-state) current $I(t)$ (Bremermann~\cite{hBremermann65}, Ch.10). We derive formulas for  $I(t)$ for some extreme and violent physical phenomena such as lightning and superconductivity which do not have counterparts in classical analysis and the Schwartz theory of distributions. We challenge the common wisdom that only the real-valued functions have the right to present physical quantities: our formulas for the electrical current $I(t)$ are often generalized functions with infinitely large values which sometimes keep their infinitely large values in time.
\end{enumerate}
\section{\bf Linear Independence in $\widehat{\mathcal E(\mathbb R)}$} \label{S: Linear Independence in widehat E(R)}

	We ask the reader to refer to (Todorov \& Vernaeve~\cite{TodVern08}) for the construction of  the algebra $\widehat{\mathcal E(\mathbb R)}$ and the field of its scalars  $\widehat{\mathbb C}$. Here is a summary of this construction:

\begin{enumerate} 
\item[1.] If $S$ is a set, then every function of the form $f: \mathcal D(\mathbb R)\to S$ will be called a \textbf{$\varphi$-net in $S$} and will be denoted often by $(f_\varphi)$, where $\varphi\in\mathcal D(\mathbb R)$ is treated as a parameter (Todorov \& Vernaeve~\cite{TodVern08}, \S3). The elements of  $\widehat{\mathcal E(\mathbb R)}$ are equivalence classes  $q(f_\varphi)$ (denoted also by $\widehat{f_\varphi}$) of $\varphi$-nets in $\mathcal E(\mathbb R)$. Similarly, the elements of  $\widehat{\mathbb C}$ are equivalence classes  $q(A_\varphi)$  (denoted also by $\widehat{A_\varphi}$) of $\varphi$-nets in $\mathbb C$ (Todorov \& Vernaeve~\cite{TodVern08}, \S4). In particular, the generalized number $s=q(R_\varphi)$ is called \emph{canonical infinitesimal}, where   
\begin{equation}\label{E: RadiusSupport}
R_\varphi=
\begin{cases}
 \sup\{||x|| : x\in\mathbb{R},\;  \varphi(x)\not= 0 \}, &\varphi\not=0,\\
1,													
&\varphi=0. 
\end{cases}
\end{equation}  
stands for the the \emph{radius of support} of
$\varphi\in\mathcal{D}(\mathbb{R})$.

\item[2.] The set of the (generalized) scalars $\widehat{\mathbb C}=\{f\in\widehat{\mathcal E(\mathbb R)}: f^\prime=0\}$ of $\widehat{\mathcal E(\mathbb R)}$ is a field and we also have $\widehat{\mathbb C}=\widehat{\mathbb R}(i)$, where $\widehat{\mathbb R}$ is the field of the real (generalized) scalars of $\widehat{\mathcal E(\mathbb R)}$. Both $\widehat{\mathbb C}$ and $\widehat{\mathbb R}$ are non-Archimedean fields; $\widehat{\mathbb C}$ is an algebraically closed field and $\widehat{\mathbb R}$ is a real closed field (Todorov \& Vernaeve~\cite{TodVern08}). If  $x\in\widehat{\mathbb C}$, then $x$  is \emph{infinitesimal} if $|x|<1/n$ for all $n\in\mathbb N$. In particular, $s$ is a positive infinitesimal, i.e. $0< s< 1/n$ for all $n\in\mathbb N$. If $x, y\in\widehat{\mathbb C}$, then we write $x\approx y$ if $x-y$ is infinitesimal. 
We denote by $\mathcal F$ the set (an integral domain) of the finite elements $x\in\widehat{\mathbb R}$ for which $|x|\leq n$ for some $n\in\mathbb N$. 

\item[3.] We denote by $\iota: \mathcal D^\prime(\mathbb R)\to\widehat{\mathcal E(\mathbb R)}$, where $\iota(T)=q(T\star\varphi)$,  the canonical embedding (of Colombeau type, Colombeau~\cite{jfCol90}) of the space of Schwartz distributions into $\widehat{\mathcal E(\mathbb R)}$ (we shall use here $\iota$ instead of $E_\mathbb R$ used in Todorov \& Vernaeve~\cite{TodVern08}, p. 222). Here $T\star\varphi$ stands for the convolution product between $T$ and $\varphi$ in the sense of theory of distributions (Vladimirov~\cite{vVladimirov}, 79-80). In particular, $\iota(\delta)=q(\varphi)$ (Todorov \& Vernaeve~\cite{TodVern08}, \S5). 

\item[4.] The elements $f$ of $\widehat{\mathcal E(\mathbb R)}$ are  pointwise functions of the type $f: \mathcal F\to\widehat{\mathbb C}$. The latter holds, in particular, for $f=\iota(T)$, where $T\in\mathcal D^\prime(\mathbb R)$. If $f\in\mathcal E(\mathbb R)$, then $\iota(f)$ is an extension of $f$ in the sense that $\iota(f)(x)=f(x)$ for all $x\in\mathbb R$. On this ground we shall often use the notation $f$ instead of the more precise $\iota(f)$ leaving the reader to figure out what we actually mean from context. For each $k\in\mathbb N_0$ the value $\iota(\delta^{(2k)})(0)$ is infinitely large number in $\widehat{\mathbb R}$, i.e. $n<|\iota(\delta^{(2k)})(0)|$ for all $n\in\mathbb N$. Also, $\iota(\delta^{(2k+1)})(0)=0$ (since $\iota(\delta)(t)$ is an even function). 

\item[5.] The algebra $\widehat{\mathcal E(\mathbb R)}$ and its scalars $\widehat{\mathbb C}$ admit also an axiomatic definition (Todorov~\cite{tdTodorovAxioms11} and Todorov~\cite{tdTodAlgebraic13}). We should mention that under some assumption 
$\widehat{\mathcal E(\mathbb R)}$  is isomorphic to the algebra of $\rho$-asymptotic functions $^\rho\mathcal E(\mathbb R)$, defined in (Oberguggenberger \& Todorov~\cite{OberTod98}) and $\widehat{\mathbb R}$ is isomorphic to Robinson's field of $\rho$-asymptotic numbers $^\rho\mathbb R$ (Robinson~\cite{aRob73}, Lightstone \& Robinson~\cite{LiRob}). We prefer to work with $\widehat{\mathcal E(\mathbb R)}$ (rather than with $^\rho\mathcal E(\mathbb R)$), only because the embedding $\iota: \mathcal D^\prime(\mathbb R)\to\widehat{\mathcal E(\mathbb R)}$ is canonical (while the similar embedding in  $^\rho\mathcal E(\mathbb R)$ depends on a fixed non-standard mollifier).
\end{enumerate}
\begin{definition}[Vector Spaces $\widehat{\rm span}(S)$]\label{D: Vector Spaces} 
\begin{enumerate} 
\item For every $S\subseteq\mathcal D^\prime(\mathbb R)$, we let $
\widehat{\rm span}(S)=\newline\big\{\sum_{n=1}^m\alpha_n \,\iota(T_n): m\in\mathbb N, T_n\in S, \alpha_n\in\widehat{\mathbb C}\big\}$, the span of $\iota(S)$ within the algebra $\widehat{\mathcal E(\mathbb R)}$.
We supply $\widehat{\rm span}(S)$ with the operations of a differential vector space over the field $\widehat{\mathbb C}$ inherited from the algebra $\widehat{\mathcal E(\mathbb R)}$.

\item  We supply $\widehat{\rm span}(\mathcal D_+^\prime(\mathbb R))$ with  the convolution product  $\star: \mathcal D_+^\prime(\mathbb R)\times\widehat{\rm span}(\mathcal D_+^\prime(\mathbb R))\to\widehat{\rm span}(\mathcal D_+^\prime(\mathbb R))$ by $T\star f=\sum_{n=1}^m\alpha_n\,\iota(T\star T_n)$, where $T\star T_n\in\mathcal D_+^\prime(\mathbb R)$ is the convolution product between $T$ and $T_n$ in the sense of the theory of distributions (Vladimirov~\cite{vVladimirov}, p. 77). 
\end{enumerate}
\end{definition}

	We observe that $\delta^{(n)}\star f= f^{(n)}$ for $n=0, 1, 2,\dots$. Notice that for the value $(T\star f)(0)\in\widehat{\mathbb C}$ we have $(T\star f)(0)=\sum_{n=1}^m\alpha_n\,\iota(T\star T_n)(0)$. Recall that $\iota(T\star T_n)(0)=q\big((T\star T_n\star\varphi)(0)\big)$, where \newline$q\big((T\star T_n\star\varphi)(0)\big)$ is the equivalence class of the $\varphi$-net  $((T\star T_n\star\varphi)(0))$ in $\mathbb C$ (Todorov~\&~Vernaeve~\cite{TodVern08}, 213). 

	In what follows $W(f_1,\dots,f_k)$ denotes the Wronskian of the functions $f_1,\dots,f_k$ and we denote by $W(f_1,\dots,f_k)(x)$ the value of  $W(f_1,\dots,f_k)$ at $x$.

\begin{lemma}[Linear Independence in $\widehat{\mathcal E(\mathbb R)}$] \label{L: Linear Independence in widehat E(R)} Let $y_1,\dots, y_k\in\mathcal E(\mathbb R)$ be solutions of the homogenous equation $P(d/dt)y=0$ for some polynomial $P\in\mathbb C[x]$ (in one variable with complex coefficients) of degree ${\rm deg}(P)=k$. Then $y_1,\dots, y_k$ are linearly independent in the $\mathbb C$-vector space $\mathcal E(\mathbb R)$ if and only if their $\iota$-images, $\iota(y_1),\dots, \iota(y_k)$, are linearly independent in the $\widehat{\mathbb C}$-vector space $\widehat{\mathcal E(\mathbb R)}$. 
\end{lemma}
\begin{proof} $\iota(y_1),\dots, \iota(y_k)$  are solutions of $P(d/dt)y=0$ in $\widehat{\mathcal E(\mathbb R)}$ since $\iota$ is a differential ring embedding of $\mathcal E(\mathbb R)$ into $\widehat{\mathcal E(\mathbb R)}$ and $\iota\big(W(y_1,\dots,y_k)\big)=W(\iota(y_1),\dots,\iota(y_k))$ (for the same reasons). Also, $W(y_1,\dots,y_k)(x)=W(\iota(y_1),\dots,\iota(y_k))(x)$ for all $x\in\mathbb R$ since each $\iota(y_n^{(i)})$ is a pointwise extension of $y_n^{(i)}$. Now, $y_1,\dots, y_k$ are linear independent in $\mathcal E(\mathbb R)$ if and only if $W(y_1,\dots,y_k)(0)\not=0$ if and only if $W(\iota(y_1),\dots,\iota(y_k))(0)\not=0$. The latter is equivalent to the linear independence of 
$\iota(y_1),\dots, \iota(y_k)$ in $\widehat{\mathcal E(\mathbb R)}$ by the usual arguments since   $(c\, \iota(y_n))^{(i)}(0)=c\, y_n^{(i)}(0)$ for all $c\in\widehat{\mathbb C}$ and all $n$ and $i$. 
\end{proof}
\section{\bf Associated Generalized Constants}\label{S:  Associated Generalized Constants}

\begin{definition}[Associated Generalized Constants]\label{D: Associated Generalized Constants} Let $P\in\mathbb C[x]$ be a polynomial in one variable of degree $k$ and let $G_P\in\mathcal D^\prime(\mathbb R)$ be the Green function of the operator $P(d/dt): \mathcal D^\prime(\mathbb R)\to \mathcal D^\prime(\mathbb R)$, i.e. the distributional solution of the initial value problem $P(d/dt)G_P(t)=\delta(t), \quad G_P(t)=0 \text{ on } (-\infty, 0)$. Let $\iota(G_P^{(n)})$ be the $\iota$-image of $G_P^{(n)}$ into $\widehat{\mathcal E(\mathbb R)}$. We say that $\iota(G_P^{(n)})(0)\in\widehat{\mathbb C}$,\; $n=0, 1, 2,\dots,k-1$, are the generalized constants associated with the operator $P(d/dt)$. 
\end{definition}

	Here are three examples - presented as lemmas - of  particular differential operators along with their generalized constants.

\begin{lemma}[Green Function of $a\frac{d}{dt}+b$] \label{L: Green Function of First Order} Let $a, b\in\mathbb R, a\not=0$ and let $g\in\mathcal D^\prime(\mathbb R)$ be the Green function of the operator $a\frac{d}{dt}+b$, i.e.
$a g^\prime+bg=\delta(t), \quad g(t)=0 \text{ on } (-\infty, 0)$. Then:
\begin{description}
\item[(i)] $g(t)=\frac{1}{a}H(t)e^{-\frac{b}{a}t}$, , where $H\in\mathcal D^\prime(\mathbb R)$ is the Heaviside distribution with kernel $h$, the unit step-function. Consequently, $g^\prime(t)=\frac{1}{a}\Big(\delta(t)-b\, g(t)\Big)$. More generally, $g^{(n)}(t)=\frac{1}{a}\sum_{k=0}^{n-1} (-1)^k(\frac{b}{a})^k\delta^{(n-1-k)}(t)+(-1)^n(\frac{b}{a})^n\, g(t)$, $n= 1, 2,\dots$.
\item[(ii)]  The only associated generalized constant is $\iota(g)(0)\in\widehat{\mathbb R}$ and  $\iota(g)(0)\approx 1/2a$.
\end{description}
\end{lemma}
\begin{proof} 

	(i) The general solution of $a y^\prime+by=\delta(t)$ in $\mathcal D^\prime(\mathbb R)$ is given by the family
$y(t, c)= ce^{-bt/a}+ \frac{1}{a}H(t)e^{-bt/a}$, where $c\in\mathbb C$. The initial condition $y(t)=0 \text{ on } (-\infty, 0)$ implies $c=0$. Thus $g(t)=\frac{1}{a}H(t)e^{-\frac{b}{a}t}$. The rest of the formulas follow by direct differentiation.

	(ii) We have $\iota(g)(0) =q(g\star\varphi)(0))=\frac{1}{a} q\big(\int_{-\infty}^0 e^{\frac{bx}{a}}\varphi(x)\, dx\big)=\frac{1}{a} q\big(\int_{-R_\varphi}^0 e^{\frac{bx}{a}}\varphi(x)\, dx\big)$ (Todorov\,\& Vernaeve~\cite{TodVern08}). On the other hand, by the mean value theorem for integration (Hobson~\cite{ewHobson1909}) for every test function $\varphi$ there exists $x_\varphi$ in $(-R_\varphi, 0)$ such that $\int_{-R_\varphi}^0 e^{\frac{bx}{a}}\varphi(x)\, dx= e^{-\frac{bR_\varphi }{a}}\int_{-R_\varphi}^{x_\varphi}\varphi(x)\, dx+\int_{x_\varphi}^0\varphi(x)\, dx$. Thus $\iota(g)(0)=\frac{1}{a}\big[q\big(e^{-\frac{b R_\varphi}{a}})\, q\big(\int_{R_\varphi}^{x_\varphi}\varphi(x)\, dx\big)+q\big(\int_{x_\varphi}^0\varphi(x)\, dx\big)\big]\approx\frac{1}{a}\, q\big(\int_{-R_\varphi}^{x_\varphi}\varphi(x)\, dx+\newline+\int_{x_\varphi}^0\varphi(x)\, dx\big) =\frac{1}{a}\, q\big(\int_{-\infty}^0\varphi(x)\, dx\big) =\frac{1}{2a}$ since $q(R_\varphi)=s$ is a positive infinitesimal, thus $q\big(e^{-\frac{bR_\varphi}{a}})=e^{-\frac{b q(R_\varphi)}{a}}=e^{-\frac{b s}{a}}\approx 1$, and $q\big(\int_{-R_\varphi}^{x_\varphi}\varphi(x)\, dx\big)$ is a finite number. 
\end{proof}
\begin{lemma}[Green Function of $a\frac{d^2}{dt^2}+b\frac{d}{dt}+c$]\label{L: Green Second Order} Let $a, b, c\in\mathbb R$, $a\not=0$ and $b^2-4ac<0$ and let $G\in\mathcal D_+^\prime(\mathbb R)$ be the Green function of  the operator $a\frac{d^2}{dt^2}+b\frac{d}{dt}+c$, i.e. $aG^{\prime\prime}+bG^\prime+cG=\delta(t),\; G(t)=0 \text{ on } (-\infty, 0)$. Then:
\begin{description}
\item[(i)] $G(t)=\frac{1}{a\,\omega}\, H(t)e^{-\alpha t}\sin{(\omega t)}$, where $\alpha=b/2a$ and $\omega=\frac{1}{2a}\sqrt{4ac-b^2}$. Consequently,  
\begin{align} 
& G^\prime(t)=\frac{1}{a\,\omega}H(t) \big[-\alpha e^{-\alpha t}\sin{\omega t}+\omega\, e^{-\alpha t}\cos{\omega t}\big]=-\alpha\, G(t)+\frac{1}{a}\,H(t)e^{-\alpha t}\cos{\omega t},\notag\\
& G^{\,\prime\prime}(t)=\frac{1}{a}\, \delta(t)-\frac{b}{a}\, G^\prime(t)-\frac{c}{a}\, G(t)=\frac{1}{a}\,\delta(t)-\frac{b}{a^2}\, H(t)\, e^{-\alpha t}\cos{\omega t}+\frac{\alpha^2-\omega^2}{a\omega}\, H(t)\, e^{-\alpha t}\sin{\omega t}\big]\notag.
\end{align}
\item[(ii)] Both generalized constants $\iota(G)(0)$ and $\iota(G^\prime)(0)$  are generalized numbers in $\widehat{\mathbb R}$ such that $\iota(G)(0)\approx 0$ and $\iota(G^\prime)(0)\approx1/2a$. More precisely, $|\iota(G)(0)|<s/2a$ and $|\iota(G^\prime)(0)-\frac{1}{2a}|<\frac{\alpha s}{2a}$, where $s=q(R_\varphi)$ is the canonical infinitesimal  of $\widehat{\mathbb R}$. 
\end{description}
\end{lemma}
\begin{proof} (i) We refer the reader to Bremermann~\cite{hBremermann65}, p. 127.

	(ii) We have $\iota(G)=q(G\star\varphi)=q\big(\frac{1}{a\, \omega}\int_{-R_\varphi}^te^{-\alpha(t-x)}\sin{\omega(t-x)}\,\varphi(x)\, dx\big)$; thus we have $\iota(G)(0)=q\big(-\frac{1}{a\,\omega}\int_{-R_\varphi}^0 e^{\alpha x}\sin{(\omega x)}\,\varphi(x)\, dx\big)$.  On the other hand, by the mean value theorem, there exists $x_\varphi\in(-R_\varphi, 0)$ such that $
-\frac{1}{a\,\omega}\int_{-R_\varphi}^0 e^{\alpha x}\,\sin{(\omega x)}\,\varphi(x)\, dx=\newline
=\frac{1}{a\, \omega} e^{-\alpha R_\varphi}\sin(\omega R_\varphi)\int_{-R_\varphi}^{x_\varphi}\varphi(x)\, dx$. Thus $|\iota(G)(0)|\leq q\big(\big|\frac{e^{-\alpha R_\varphi}\sin(\omega R_\varphi)}{a\, \omega}\big|\int_{-R_\varphi}^{x_\varphi}|\varphi(x)|\, dx\big)\leq \newline\frac{e^{-\alpha s}\sin(\omega s)}{a\,\omega}q\big(\int_{-\infty}^{0}|\varphi(x)|\, dx\big)
=\frac{e^{-\alpha s}\sin(\omega s)}{a\,\omega}\big|\int_{-\infty}^{0}|\delta(x)|\, dx\leq\frac{e^{-\alpha s}\sin(\omega\, s)}{a\,\omega}\,\frac{1}{2}=\frac{e^{-\alpha s}\sin(\omega s)}{2a\,\omega}<\frac{s}{2a}$. To show that $\iota(G^\prime)(0)\approx1/2a$, it suffices to show that $\iota\big(\,H(t)e^{-\alpha t}\cos{\omega t}\big)(0)\approx1/2$. Indeed, $\iota\big(H(t)e^{-\alpha t}\cos{\omega t}\big)=q\big((H(t)e^{-\alpha t}\cos{\omega t})\star\varphi\big)=q\big(\int_{-\infty}^te^{-\alpha (t-x)}\cos{\omega (t-x)}\varphi(x)\, dx\big)$. Thus $\iota\big(\,H(t)e^{-\alpha t}\cos{\omega t}\big)(0)=q\big(\int_{-\infty}^0e^{\alpha x}\cos{\omega  x}\,\varphi(x)\, dx\big)=q\big(\int_{-R_\varphi}^0e^{\alpha x}\cos{\omega  x}\,\varphi(x)\, dx\big)$. On the other hand,  by the mean value theorem for integration (Hobson~\cite{ewHobson1909}), there exists $x_\varphi\in(-R_\varphi, 0)$ such that $\int_{-R_\varphi}^0e^{\alpha x}\cos{\omega x}\, \varphi(x)\, dx
=e^{-\alpha R_\varphi}\cos{(\omega R_\varphi)}\int_{-R_\varphi}^{x_\varphi}\varphi(x)\, dx+\int_{x_\varphi}^0\varphi(x)\, dx$. Thus \newline $\iota\big(\,H(t)e^{-\alpha t}\cos{\omega t}\big)(0)=q\big(e^{-\alpha R_\varphi}\cos{\omega R_\varphi}\big)q\big(\int_{-R_\varphi}^{x_\varphi}\varphi(x)\, dx)+q\big(\int_{x_\varphi}^0\varphi(x)\, dx\big)\approx q\big(\int_{-R_\varphi}^{x_\varphi}\varphi(x)\, dx)+q\big(\int_{x_\varphi}^0\varphi(x)\, dx\big)= q\big(\int_{-R_\varphi}^0\varphi(x)\, dx\big)=1/2$ since $q\big(e^{-\alpha R_\varphi}\cos{(\omega R_\varphi)}\big)=e^{-\alpha s}\cos{(\omega s)}\approx 1$.
\end{proof}
\begin{lemma}[Green Function of $\frac{d^2}{dt^2}+\omega^2$] \label{L: Green Function for R=0} Let $\omega\in\mathbb R_+$ and let $G_0\in\mathcal D_+^\prime(\mathbb R)$ be the Green function of  the operator $\frac{d^2}{dt^2}+\omega^2$, i.e. $G_0^{\prime\prime}+\omega^2G_0=\delta(t),\; G_0(t)=0 \text{ on } (-\infty, 0)$. Then:
\begin{description}
\item[(i)] $G_0(t)=\frac{1}{\omega}H(t)\sin{(\omega t)}$, where (as before) $H\in\mathcal D^\prime(\mathbb R)$ is the Heaviside distribution with kernel $h$, the unit step-function. Consequently,  $G_0^\prime(t)=H(t)\cos{(\omega t)}$, $G_0^{\prime\prime}(t)=\delta(t)-\omega^2G_0(t)$ and $G_0^{\prime\prime\prime}(t)=\delta^\prime(t)-\omega^2G_0^\prime(t)$. More generally,
\begin{align}
&G_0^{(2n+2)}(t)=\sum_{k=0}^{n} (-1)^k\omega^{2k}\delta^{(2n-2k)}(t)+(-1)^{n+1}\omega^{2(n+1)}\,G_0(t),\notag\\
&G_0^{(2n+3)}(t)=\sum_{k=0}^{n} (-1)^k\omega^{2k}\delta^{(2n+1-2k)}(t)+(-1)^{n+1}\omega^{2(n+1)}\,G_0^\prime(t),\notag
\end{align}
for $n= 0, 1, 2,\dots$.
\item[(ii)] The generalized constants $\iota(G_0)(0)$ and $\iota(G_0^\prime)(0)$ are generalized numbers in $\widehat{\mathbb R}$ such that $\iota(G_0)(0)\approx 0$ and $\iota(G_0^\prime)(0)=1/2$. More precisely, $\iota(G_0)(0)\leq s/2$, where $s=q(R_\varphi)$ is the canonical infinitesimal  of $\widehat{\mathbb R}$. Consequently, we have
\begin{align}
&\iota(G_0^{(2n+2)})(0)=\sum_{k=0}^{n} (-1)^k\omega^{2k}\iota(\delta^{(2n-2k)})(0)+(-1)^{n+1}\omega^{2(n+1)}\,\iota(G_0)(0),\notag\\
&\iota(G_0^{(2n+3)})(0)=-\frac{(-1)^{n}\omega^{2(n+1)}}{2}. \notag
\end{align}
\end{description}
\end{lemma}
\begin{proof}  Except for the strict equality $\iota(G_0^\prime)(0)=1/2$, the results follow from the previous lemma for $a=1, b=0$ and $c=\omega^2$. We have $\iota(G_0^\prime)(t)=q\big(\int_{-\infty}^t\cos{\omega(t-x)}\varphi(x)\, dx\big)$; thus we have $\iota(G_0^\prime)(0)=q\big(\int_{-\infty}^0\cos{(\omega x)}\,\varphi(x)\, dx\big)=q\big(\frac{1}{2}\int_{-\infty}^\infty\cos{(\omega x)}\,\varphi(x)\, dx\big)
=\frac{1}{2}\, \int_{-\infty}^\infty \iota(\delta)(x)\cos(\omega x)\, dx=\frac{1}{2}$. In the last two formulas we have taken into account that  $\delta^{(n)}(0)=0$ for odd $n$.
\end{proof}

\section{\bf Steady-State Solutions in $\widehat{\mathcal E(\mathbb R)}$}\label{S: Steady-State Solutions in hatR(R)}
\begin{theorem}[First Order ODE]\label{T: First Order ODE} Let $a, b\in\mathbb R, a\not=0$ and let $g(t)=\frac{1}{a}H(t)e^{-\frac{b}{a}t}$ be the Green function of the operator $a\frac{d}{dt}+b$ (Lemma~\ref{L: Green Function of First Order}). Then for every $f\in\widehat{\rm span}(\mathcal D_+^\prime(\mathbb R))$ (Definition~\ref{D: Vector Spaces}) the initial value problem $ay^\prime+by=f(t),\, y(0)=0$, has a unique (steady-state) solution $y$ in $\widehat{\mathcal E(\mathbb R)}$ (\ref{E: Steady-State}), given by the formula
\begin{equation}\label{E: 1st Answer}
y(t)=(g\star f)(t)-(g\star f)(0)\, e^{-\frac{b}{a}t},
\end{equation} 
where $g\star f\in\widehat{\rm span}(\mathcal D_+^\prime(\mathbb R))$ and  $(g\star f)(0)\in\widehat{\mathbb R}$.
\end{theorem}

\begin{proof} We have $f=\sum_{n=1}^m\alpha_n\,\iota(T_n)$ for some $m\in\mathbb N,\; T_n\in\mathcal D_+^\prime(\mathbb R)$ and $\alpha_n\in\widehat{\mathbb C}$ by assumption. Next, the general solution of the equation $ay^\prime+by=T_n$ in $\mathcal D^\prime(\mathbb R)$ is given by  the family $y_n(t, c)= (g\star T_n)(t)+c\, e^{-\frac{b}{a}t}$, where $c\in\mathbb C$. Consequently, the general solution of the equation $ay^\prime+by=\iota(T_n)$ in $\widehat{\mathcal E(\mathbb R)}$ is given by the family $y_n(t, c)=\iota(g\star T_n)(t)+c\, e^{-\frac{b}{a}t}$, where $c\in\widehat{\mathbb C}$. The initial condition $y_n(0, c)=0$ (treated as a equation for $c$ in $\widehat{\mathbb C}$) produces $c_n=-\iota(g\star T_n)(0)$. Thus the solution of the initial value problem $ay^\prime+by=\iota(T_n), \; y(0)=0$, in $\widehat{\mathcal E(\mathbb R)}$ is $y_n(t, c_n)= \iota(g\star T_n)(t)-\iota(g\star T_n)(0)\, e^{-\frac{b}{a}t}$. The latter implies that $y=\sum_{n=1}^m\alpha_n\, y_n(t, c_n)$ is the solution of $ay^\prime+by=f(t),\, y(0)=0$ by linearity. Thus $y= \sum_{n=1}^m\alpha_n\big(\iota(g\star T_n)(t)-\iota(g\star T_n)(0)\, e^{-\frac{b}{a}t}\big)=(g\star f)(t)-(g\star f)(0)\, e^{-\frac{b}{a}t}$ as required. 
\end{proof}
\begin{theorem}[Higher Order ODE]\label{T: Higher Order ODE} Let $P\in\mathbb C[x]$ and ${\rm deg}(P)=k$. Let  $f\in\widehat{\rm span}(\mathcal D_+^\prime(\mathbb R))$ (Definition~\ref{D: Vector Spaces}). Then the initial value problem (\ref{E: Steady-State}) has a unique (steady-state) solution $y$ in the algebra $\widehat{\mathcal E(\mathbb R)}$ given by the formula $y= \sum_{n=1}^kc_n\,\iota(y_n) +G\star f$, where $y_1,\dots, y_k$ are linearly independent solutions of the homogenous equation $P(d/dt)y=0$ in $\mathcal E(\mathbb R)$, $G\in\mathcal D_+^\prime(\mathbb R)$ is the Green function of the operator $P(d/dt)$, the convolution product $G\star f\in\widehat{\rm span}(\mathcal D_+^\prime(\mathbb R))$ is in the sense of Definition~\ref{D: Vector Spaces} and the generalized constants $c_i\in\widehat{\mathbb C}$ are determined by
\begin{equation}\label{E: Wronskians}
c_n= -\frac{W\big(y_1,\dots,y_{n-1}, G\star f, y_{n+1},\dots, y_k\big)(0)}{W(y_1,\dots,y_{n-1},y_n, y_{n+1},\dots, y_k)(0)},\quad n=1,\dots,k.
\end{equation}
\end{theorem}

\begin{proof} We observe that for each $n$ the family $Y_n(t, C_1,\dots, C_k)=\sum_{i=1}^k C_i\,y_i+G\star T_n$, where $C_i\in\mathbb C$, presents the general solution of the equation $P(d/dt)y=T_n$ in $\mathcal D^\prime(\mathbb R)$. Consequently, the family $y_n(t, c_1,\dots, c_k)=\sum_{i=1}^k c_i\,\iota(y_i)+\iota(G\star T_n)$, where $c_i\in\widehat{\mathbb C}$, presents the general solution of the equation $P(d/dt)y=\iota(T_n)$ in $\widehat{\mathcal E(\mathbb R)}$ since $\iota(y_1),\dots, \iota(y_k)$ are solutions of the homogenous equation $P(d/dt)y=0$ in $\widehat{\mathcal E(\mathbb R)}$ which are linearly independent in $\widehat{\mathcal E(\mathbb R)}$ by Lemma~\ref{L: Linear Independence in widehat E(R)}. Thus the family $y(t, c_1,\dots, c_k)=\sum_{n=1}^k c_n\,\iota(y_n)+G\star f$, where $c_i\in\widehat{\mathbb C}$, presents the general solution of the equation $P(d/dt)y=G\star f$ in $\widehat{\mathcal E(\mathbb R)}$ by the linearity of the operator $P(d/dt)$. The initial conditions $y(0)=\dots=y^{(k-1)}(0)=0$ leads to the system of linear equations for the constants $c_1,\dots, c_k$ in $\widehat{\mathbb C}$\,:
\[
\begin{cases}
&\sum_{n=1}^k c_n\, y_n(0)=-(G\star f)(0),\notag\\
&\dots\dots\dots\dots\dots\dots\dots\dots\dots\dots\notag\\
&\sum_{n=1}^k c_n\, y_n^{(k-1)}(0)=-(G\star f)^{(k-1)}(0),\notag
\end{cases}
\]
since $\iota(c_n\,y_n^{(i)})(0)=c_n\,y_n^{(i)}(0)$ for each $i$ and $n$. The Crammer rule produces the formulas  (\ref{E: Wronskians}) for constants $c_i$ in the the formula $y= G\star f+\sum_{n=1}^kc_n\,\iota(y_n)$.
\end{proof}
\begin{remark}[Simplified Notation] We shall often drop the $\iota$'s by letting $y_i=\iota(y_i)$ on the ground that $\iota(y_i)$ are pointwise extensions of $y_i$. In this simplified notation the above formula will be written simply as $y= \sum_{n=1}^kc_n\,y_n+G\star f$.
\end{remark}


\section{\bf LRC-Electrical Circuit: Lightning, Lightning Rods and Superconductivity}

	We present several examples of steady-state solutions to first and second order ordinary differential equations with constant coefficients. We supply these initial value problems with physical interpretation borrowed from Kirchoff law for an electrical $LRC$-circle. Some reader might prefer to ignore physics and focus on mathematics. 

\begin{definition}[Lightning, Lightning Rods and Superconductivity]\label{D: Lightning and Superconductivity} Let $V(t)\in\widehat{\rm span}(\mathcal D_+^\prime(\mathbb R))$ (Definition~\ref{D: Vector Spaces}). Let $L, R, C\in\mathbb R$ such that $L\geq 0, R\geq 0$, $C>0$ and $L^2+R^2\not= 0$. Let $I(t)\in\widehat{\mathcal E(\mathbb R)}$ be the solution of the IVP: $LI^{\prime\prime}+R\, I^\prime+\frac{1}{C}\,I= V^\prime(t),\, I(0)=0,\,  LI^\prime(0)=0$. Let $A\in\widehat{\mathbb R}_+$. Then:
\begin{enumerate}
\item[1.] We associate the above initial value problem with the  \textbf{Kirchoff law} for an electrical $LRC$-circle with inductance $L$, resistance $R$, capacitance $C$, driving electromotive force (voltage) $V(t)$ and (steady state) \textbf{current} $I(t)$.  

\item[2.] $V(t)=A\iota(H)(t)$ is called \textbf{switch with amplitude} $A$: 

\item[3.] Let $n\in\mathbb N_0$. Then $V(t)=A\,\iota(\delta^{(n)})(t)$ is called \textbf{lightning of order $n$ with amplitude} $A$.

\item[4.] The operator $L\frac{d^2}{dt^2}+R\frac{d}{dt}+\frac{1}{C}$ is called a \textbf{lightning rod} if $L, R$ and $C$ are infinitesimals ($C$ must be a non-zero infinitesimal).

\item[5.] If $R$ is infinitesimal ($R=0$ is not excluded), we talk about \textbf{superconductivity}.
\end{enumerate}
\end{definition}

\begin{lemma}[Steady-State Current]\label{L: Steady-State Current} Let $L, R, C$ and $V$ be as in Definition~\ref{D: Lightning and Superconductivity}.
\begin{description}
\item[(i)] If $L=0$, then the Kirchoff law reduces to $R\, I^\prime+\frac{1}{C}\,I= V^\prime(t),\, I(0)=0$
and its solution in $\widehat{\mathcal E(\mathbb R)}$ is given by the formula $I(t)=(g\star V^\prime)(t)-(g\star V^\prime)(0)\, e^{-\frac{t}{RC}}$, where $g(t)=\frac{1}{R}\, H(t)\, e^{-\frac{t}{RC}}$.

\item[(ii)] Let $R^2-\frac{4L}{C}<0$.  Then the Kirchoff law reduces to $LI^{\prime\prime}+R\, I^\prime+\frac{1}{C}\,I= V^\prime(t),\, I(0)=I^\prime(0)=0$,
and its solution in $\widehat{\mathcal E(\mathbb R)}$ is 
\begin{align}
I(t)=-(G\star V^\prime)(0)\,e^{-\frac{R}{2L}t}\,\cos{\omega t}-&\frac{R\, (G\star V^\prime)(0)+2L(G^\prime\star V^\prime)(0)}{2L\omega}\,e^{-\frac{R}{2L}t}\,\sin{\omega t}+\notag\\
&+(G\star V^\prime)(t),\notag
\end{align}
where $\omega=\frac{1}{2L}\sqrt{\frac{4L}{C}-R^2}$,    $\iota(G)(0)\approx 0$ and $\iota(G^\prime)(0)\approx 1/2L$ (Lemma~\ref{L: Green Second Order}) and $G(t)=\frac{1}{L\omega}\, H(t)\,e^{-\frac{R}{2L}t}\,\sin{\omega t}$ (consequently, $G^\prime(t)=-\frac{R}{2 L}G(t)+\frac{1}{L}\, H(t)\,e^{-\frac{R}{2L}t}\,\cos{\omega t}$, and $G^{\prime\prime}=\frac{1}{L}\delta(t)-\frac{R}{L}\,G^\prime(t)-\frac{1}{LC}\, G(t)=\frac{1}{L}\,\delta(t)-\frac{R}{L^2}\, H(t)\,e^{-\frac{R}{2L}t}\,\cos{\omega t}+\frac{1}{L\omega}(\frac{R^2}{4L^2}-\omega^2)\, H(t)\,e^{-\frac{R}{2L}t}\,\sin{\omega t}$).

\item[(iii)] Let $R=0$ (\textbf{superconductivity}) and $L, C\in\mathbb R_+$. Then the Kirchoff law reduces to $L\, I^{\prime\prime}+\frac{1}{C}\,I= V^\prime(t),\, I(0)=I^\prime(0)=0$, and its solution in $\widehat{\mathcal E(\mathbb R)}$ is given by:
\begin{align}
I(t)=&-\frac{\, (G_0\star V^\prime)(0)}{L}\,\cos{\omega t}-\frac{\,(G_0^\prime\star V^\prime)(0)}{L\omega}\,\sin{\omega t}
+\frac{1}{L}\,(G_0\star V^\prime)(t),\notag
\end{align}
where $\omega=\frac{1}{\sqrt{LC}}$,  $\iota(G_0)(0)\approx 0$, $\iota(G_0^\prime)(0)=1/2$ (Lemma~\ref{L: Green Function for R=0}) and $G_0(t)=\frac{1}{\omega}\, H(t)\,\sin{\omega t}$ (consequently, $G_0^\prime(t)= H(t)\,\cos{\omega t}$, and $G_0^{\prime\prime}=\delta(t)-\omega^2\, G_0(t)=\delta(t)-\omega\, H(t)\,\sin{\omega t}$).
\end{description}
\end{lemma}
\begin{proof} (i) is identical to Theorem~\ref{T: First Order ODE} for $f(t)=V^\prime(t)$.

	(ii) $W(e^{-\frac{R}{2L} t}\cos{\omega t}, e^{-\frac{R}{2L} t}\sin{\omega t})(0)=\omega$,\quad  $W((G\star V^\prime), e^{-\frac{R}{2L} t}\sin{\omega t})(0)=\omega\, (G\star V^\prime)(0)$ and $W(e^{-\frac{R}{2L} t}\cos{\omega t}, (G\star V^\prime))(0)=\frac{R}{2L}\, (G\star V^\prime)(0)+(G^\prime\star V^\prime)(0)$. Thus $c_1=-(G\star V^\prime)(0)$ and $c_2=-\frac{1}{\omega}\, \big(\frac{R}{2L}\, (G\star V^\prime)(0)+(G^\prime\star V^\prime)(0)\big)=-\frac{R(G\star V^\prime)(0)+2L(G^\prime\star V^\prime)(0)}{2L\omega}$ and the result follows by Theorem~\ref{T: Higher Order ODE}.

	(iii) is a particular case of (ii) for $R=0$.
\end{proof}

\begin{example}[$L=0$ and Switch] \label{Ex: Switch and L=0} Let $V(t)=A\iota(H)(t)$ (\textbf{switch}), where $A\in\widehat{\mathbb R}_+$. The corresponding initial value problem is $R\, I^\prime+\frac{1}{C}\,I= A\, \iota(\delta)(t),\, I(0)=0$. We have $G\star f=g\star (A\,\iota(\delta))= A\,\iota(g)$ and Lemma~\ref{L: Steady-State Current} produces the steady-state solution $I(t) = A\big[\iota(g)(t)-\iota(g)(0)\,e^{-\frac{t}{RC}}\big]$  or, equivalently, $I(t) = A\big[\frac{1}{R}\,\iota(H(t)e^{-\frac{t}{RC}})-\iota(g)(0)\,e^{-\frac{t}{RC}}\big]$, where  $\iota(g)(0)$ is a generalized constant in $\widehat{\mathbb R}$ which is infinitely close to $1/2R$.
\end{example}
\begin{example}[$L=0$ and Lightning of Zero Order] \label{Ex: Lightning of First Order with L=0} Let $V(t)=A\iota(\delta)(t)$ (\textbf{lightning of zero order}). The corresponding initial value problem is $R\, I^\prime+\frac{1}{C}\,I=A\,\iota(\delta^\prime)(t),\, I(0)=0$.
We have $G\star f=g\star (A\,\iota(\delta^\prime))= A\,\iota(g^\prime)=\frac{A}{R}\, [\iota(\delta)(t)-\frac{1}{C}\iota(g)(t)]$ and Lemma~\ref{L: Steady-State Current} produces the steady-state solution
\[
I(t) = A\,\Big[-\frac{\iota(\delta)(0)}{R}\,e^{-\frac{t}{RC}}+\frac{\iota(g)(0)}{RC}\,e^{-\frac{t}{RC}}-\frac{1}{R^2C}\,\iota\big(H(t)e^{-\frac{t}{RC}}\big)+\frac{1}{R}\,\iota(\delta)(t)\Big].
\]
where $\iota(g)(0)$ and  $\iota(\delta)(0)$  are generalized constants in $\widehat{\mathbb R}$ such that $\iota(g)(0)\approx 1/2R$ and $\iota(\delta)(0)$ is an infinitely large constant.
\end{example}

\begin{example}[$L=0$ and Lightning of Zero Order, Superconductivity] \label{Ex: Lightning of First Order, Superconductivity and L=0} Let $a, b\in\mathbb R_+$. We are looking for the steady state current $I(t)$ in an electrical $LRC$-circle with inductance $L=0$, infinitesimal resistance $R=a/\iota(\delta)(0)$ (\textbf{superconductivity}) , infinitely large capacitance $C=\iota(\delta)(0)/b$ and driving electromotive force $V(t)=\iota(\delta)(t)$ (\textbf{lightning of zero order}), i.e. $\frac{a}{\delta(0)}\, I^\prime+\frac{b}{\delta(0)}\,I=\iota(\delta^\prime)(t),\, I(0)=0$, or equivalently, $
a\, I^\prime+b\,I=\iota(\delta)(0)\,\iota(\delta^\prime)(t),\, I(0)=0$.
We apply the Lemma~\ref{L: Steady-State Current} for $V^\prime=\iota(\delta)(0)\, \iota(\delta^\prime)$ and the result is $g\star V^\prime=\frac{\iota(\delta)(0)}{a}\big(\iota(\delta)(t)-b\,\iota(g)(t)\big)$. Thus the steady state solution is
\[
I(t) = \frac{\iota(\delta)(0)}{a}\,\Big[-\iota(\delta)(0)\,e^{-\frac{b}{a}t}+b\,\iota(g)(0)\,e^{-\frac{b}{a}t}- \frac{b}{a}\,\iota\big(H(t)\,e^{-\frac{b}{a}t}\big) +\iota(\delta)(t)\Big].
\]
\end{example}

\begin{example}[$L=0$ and Lightning of $n$-th Order] \label{Ex: Lightning of n-th Order} Let  $A\in\widehat{\mathbb R}_+$ and $V(t)=A\iota(\delta^{(n)})(t)$ (\textbf{lightning of $n$-th order}).  The corresponding initial value problem is $R\, I^\prime+\frac{1}{C}\,I=A\,\iota(\delta^{(n+1)})(t),\, I(0)=0$. We have $(G\star f)(t)=A\,\iota(g^{(n+1)})(t)$ and with the help of Lemma~\ref{L: Green Function of First Order} we obtain:
\begin{align}
I(t) = &\frac{A}{R}\,\Big[\sum_{k=0}^{n}\frac{(-1)^{k+1}\, \iota(\delta^{(n-k)})(0)}{(RC)^k}\,e^{-\frac{t}{RC}}+\frac{(-1)^{n}\, \iota(g)(0)}{R^{n}C^{n+1}}\,e^{-\frac{t}{RC}}+\notag\\
&+\frac{(-1)^{n+1}}{(RC)^{n+1}}\iota\big(H(t)e^{-\frac{t}{RC}}\big)+\sum_{k=0}^{n}\frac{(-1)^k}{(RC)^k}\iota(\delta^{(n-k)})(t)\notag
\Big],\notag
\end{align}
where (as before) $\iota(g)(0)\approx 1/2R$ and $\iota(\delta^{2n)})(0)$ are infinitely large numbers and $\iota(\delta^{2n+1)})(0)=0$ for $n=1, 2,\dots$. 
\end{example}
\begin{example} Let $\frac{4L}{C}-R^2>0$ and let $A\in\widehat{\mathbb R}_+$. 
\begin{enumerate}
\item[1.] Let $V(t)=A\, \iota(H)(t)$ (\textbf{switch}). The solution of $LI^{\prime\prime}+R I^\prime+\frac{1}{C}I= A\, \iota(\delta)(t),\, I(0)=I^\prime(0)=0$ is given by 
\begin{align}
I(t)=A\Big[&-\iota(G)(0)\, e^{-\frac{R}{2L}t}\cos{\omega t}-\Big(\frac{\iota(G^\prime)(0)}{\omega}+\frac{R\,\iota(G)(0)}{2L\omega}\Big)\, e^{-\frac{R}{2L}t}\sin{\omega t} +\notag\\
&+\frac{1}{L\omega}\iota(H(t)\, e^{-\frac{R}{2L}t}\sin{\omega t})\Big]. \notag
\end{align}
\notag

\item[2.] Let $V(t)=A\, \iota(\delta(t))$ (\textbf{lightning of zero order}). The solution of $LI^{\prime\prime}+R\, I^\prime+\frac{1}{C}\,I= A\, \iota(\delta^\prime)(t),\, I(0)=I^\prime(0)=0$, is given by 
\begin{align}
I(t)=A\Big[&-\iota(G^\prime)(0)\, e^{-\frac{R}{2L}t}\cos{\omega t}+\notag\\
&+\big(-\frac{\iota(\delta)(0)}{L\omega}+\frac{R\, \iota(G^\prime)(0)}{2L\omega}+\frac{\iota(G)(0)}{LC\omega}\Big)\, e^{-\frac{R}{2L}t}\sin{\omega t} +\notag\\
&+\frac{1}{L}\,\iota\big(H(t)\,  e^{-\frac{R}{2L}t}\cos{\omega t}-\frac{R}{2L^2\omega}\,\iota\big(H(t)\,  e^{-\frac{R}{2L}t}\sin{\omega t}\big)\big)\Big].
\end{align}

\item[3.] Let $V(t)=A\, \iota(\delta^\prime)(t)$ (\textbf{lightning of first order}). The solution of $LI^{\prime\prime}+R\, I^\prime+\frac{1}{C}\,I= A\, \iota(\delta^{\prime\prime})(t),\, I(0)=I^\prime(0)=0$,  is given by 
\begin{align}
I(t)=A\Big[&\Big(-\frac{\iota(\delta)(0)}{L}+\frac{R\,\iota(G^\prime)(0)}{L}+\frac{\, \iota(G)(0)}{LC}\Big)\, e^{-\frac{R}{2L}t}\cos{\omega t}+\notag\\
&\Big(\frac{R\,\iota(\delta)(0)}{2L^2\omega}+\frac{(2L-R^2C)\, \iota(G^\prime)(0)}{2L^2C\omega}-\frac{R\,\iota(G)(0)}{2L^2C\omega}\Big)\, e^{-\frac{R}{2L}t}\sin{\omega t} +\notag\\
&+\frac{\iota(\delta)(t)}{L}-\frac{R}{L^2}\,\iota(H(t)\, e^{-\frac{R}{2L}t}\cos{\omega t}) +\frac{R^2C-2L}{2L^3C\omega}\,\iota(H(t)\, e^{-\frac{R}{2L}t}\sin{\omega t})\Big],
\end{align}
where $\omega=\frac{1}{2L}\sqrt{\frac{4L}{C}-R^2}$, \,$\iota(G)(0)\approx 0$ and $\iota(G^\prime)(0)\approx 1/2L$ (Lemma~\ref{L: Green Second Order}). Notice that $\iota(\delta^\prime)(0)=0$ (Todorov {\rm \&} Vernaeve~\cite{TodVern08}).
\end{enumerate}
\end{example}
\begin{example}[Superconductivity]\label{Ex: Superconductivity and a Switch} Let $R=0$ (\textbf{superconductivity}), $\omega=1/\sqrt{LC}$, $A\in\widehat{\mathbb R}_+$ and $G_0(t)=\frac{1}{\omega}H(t)\sin{(\omega t)}$. Recall that $\iota(G_0)(0)\approx 0$ and $\iota(G^\prime_0)(0)= 1/2$ (Lemma~\ref{L: Green Function for R=0}). 

\begin{enumerate}

\item If $V(t)=A\, \iota(H)(t)$ (\textbf{switch}), then the solution of $L\,I^{\prime\prime}+\frac{1}{C}\,I=A\,\iota(\delta)(t), \, I(0)=I^\prime(0)=0$, in $\widehat{\mathcal E(\mathbb R)}$ is $
I(t)=\frac{A}{L}\Big[\frac{1}{\omega}\, \iota(H(t)\sin{\omega t})-\frac{1}{2\omega}\sin{\omega t}-\iota(G_0)(0)\cos{\omega t}\Big]$.
\item  If $V(t)=A\, \iota(\delta)(t)$ (\textbf{lightning of zero order}), then the solution of $L\,I^{\prime\prime}+\frac{1}{C}\,I=A\, \iota(\delta^\prime)(t), \quad I(0)=I^\prime(0)=0$, in $\widehat{\mathcal E(\mathbb R)}$ is  $I(t)=\frac{A}{L}\Big[-\frac{\delta(0)}{\omega}\sin{\omega t} + \iota(H(t)\cos{\omega t})-\frac{1}{2}\cos{\omega t}+\omega\, \iota(G_0)(0)\sin{\omega t}\Big]$. 

\item  If $V(t)=A\, \iota(\delta^{\prime})(t)$ (\textbf{lightning of first order}), then the solution of $L\,I^{\prime\prime}+\frac{1}{C}\,I=A\, \iota(\delta^{\prime\prime})(t), \quad I(0)=I^\prime(0)=0$ in $\widehat{\mathcal E(\mathbb R)}$ is $I(t)=\frac{A}{L}\Big[-\iota(\delta)(0)\cos{\omega t} + \frac{\omega}{2}\sin{\omega t}-\omega\, \iota(H(t)\sin{\omega t})+\omega^2\, \iota(G_0)(0)\cos{\omega t}+\iota(\delta)(t)\Big]$.

\item If $V(t)=A\, \iota(\delta^{(2n+1)})(t)$ for some $n\in\mathbb N_0$ (\textbf{lightning of odd order}), then the solution of $L\,I^{\prime\prime}+\frac{1}{C}\,I=A\, \iota(\delta^{(2n+2)})(t), \quad I(0)=I^\prime(0)=0$, in $\widehat{\mathcal E(\mathbb R)}$ is 
\begin{align}
&I(t)=\frac{A}{L}\Big[-\sum_{k=0}^n(-1)^k\omega^{2k}\iota(\delta^{(2n-2k)})(0)\cos{\omega t}+\notag\\
&+(-1)^n\omega^{2n+1}\big(\frac{1}{2}\sin{\omega t}-\iota(H(t)\sin{\omega t})\big) +\notag\\
&+(-1)^n\omega^{2(n+1)}\iota(G_0)(0)\cos{\omega t}+\sum_{k=0}^n(-1)^k\omega^{2k}\iota(\delta^{(2n-2k)})(t)\Big].\notag
\end{align}

\item  If $V(t)=A\, \iota(\delta^{(2n+2)})(t)$  for some $n\in\mathbb N_0$ (\textbf{lightning of even order}), then the solution of $L\,I^{\prime\prime}+\frac{1}{C}\,I=A\, \iota(\delta^{(2n+3)})(t), \, I(0)=I^\prime(0)=0$, in $\widehat{\mathcal E(\mathbb R)}$ is 
\begin{align}
&I(t)=\frac{A}{L}\Big[-\sum_{k=0}^{n+1}(-1)^k\omega^{2k-1}\iota(\delta^{(2n+2-2k)})(0)\,\sin{\omega t}+\notag\\
&+(-1)^n\omega^{2(n+1)}\big(\frac{1}{2}\cos{\omega t}-\iota(H(t)\cos{\omega t})\big) -\notag\\
&-(-1)^n\omega^{2n+3}\iota(G_0)(0)\sin{\omega t}+\sum_{k=0}^n(-1)^k\omega^{2k}\iota(\delta^{(2n+1-2k)})(t)\Big].\notag
\end{align}
 The derivation of the formula is similar to the one in the previous example and leave it to the reader. 

\end{enumerate}
\end{example}

	Our last example is outside the the scope of Theorem~\ref{T: First Order ODE} and Theorem~\ref{T: Higher Order ODE}.
 
\begin{example}[Lightning Rod under Lightning] Let $\lambda=|\ln{s}|=q(\lambda_\varphi)$, where $\lambda_\varphi=-\ln{R_\varphi}$ (\ref{E: RadiusSupport}) and $A=q(A_\varphi)\in\widehat{\mathbb R}_+$. Notice that $\lambda$ is a positive infinitely large number, i.e. $n<\lambda$ for all $n\in\mathbb N$. We associate the initial value problem $\frac{1}{\lambda^2}I^{\prime\prime}+\frac{2}{\lambda}I^\prime +\lambda^2I=A\,\iota(\delta^\prime)(t),\; I(0)=I^\prime(0)=0$, with the LRC-electrical circle (\textbf{lightning rod}) with $L=C=1/\lambda^2$ and $R=2/\lambda$ and electromotive force $V(t)=A\, \iota(\delta)(t)$ (\textbf{under lightning of zero order with amplitude} $A$). We prefer the form $I^{\prime\prime}+2\lambda I^\prime +\lambda^4I=A\, \lambda^2\,\iota(\delta^\prime)(t),\; I(0)=I^\prime(0)=0$. For any test function $\varphi\in\mathcal D(\mathbb R)$ the initial value problem  $I^{\prime\prime}+2\lambda_\varphi\, I^\prime +\lambda_\varphi^4 \,I=A_\varphi\,\lambda_\varphi^2\,\varphi^\prime(t),\; I(0)=I^\prime(0)=0$, has a unique classical solution in $\mathcal E(\mathbb R)$ given by the formula
\begin{align}
I_\varphi(t)=A_\varphi\, \lambda_\varphi^2\Big[&\int_0^t e^{-\lambda_\varphi(t-x)}\cos{\omega_\varphi (t-x)}\, \varphi(x)\, dx- \frac{\lambda_\varphi}{\omega_\varphi}\, \int_0^t e^{-\lambda_\varphi(t-x)}\sin{\omega_\varphi (t-x)}\, \varphi(x)\, dx-\notag\\
&-\frac{\varphi(0)}{\omega_\varphi}\; e^{-\lambda_\varphi \,t}\,\sin{\omega_\varphi t}\,\Big],\notag
\end{align}
where $\omega_\varphi=\lambda_\varphi^2\sqrt{1-1/\lambda_\varphi^2}$.  This formula is obtained by Laplace transform method, but the reader is invited to verify its validity by direct differentiations and evaluations in the framework of $\mathcal E(\mathbb R)$. Notice that $e^{-\lambda_\varphi t}$ is a moderate $\varphi$-net since $e^{-\lambda_\varphi t}=(R_\varphi)^t\leq (R_\varphi)^m$ for $m=\min\{n\in\mathbb Z: n\leq t<n+1\}$ (Todorov {\rm \&} Vernaeve~\cite{TodVern08}, \S 4). Consequently $(I_\varphi(t))$ is a moderate net. We also observe the the above formula implies that $(I_\varphi(t))$ does not depend on the choice of the representative of the generalized number $\lambda$ modulo the ideal $\mathcal N(\mathcal E(\mathbb R)^\mathcal D)$ (Todorov {\rm \&} Vernaeve~\cite{TodVern08}, \S 4). Thus the generalized function $I(t)=q(I_\varphi(t))\in\widehat{\mathcal E(\mathbb R)}$ is the solution we are looking for. We present the final result in a ``distributional-like'' form:
\begin{align}
I(t)=A\,\lambda^2\Big[&-\frac{\iota(\delta)(0)}{\omega}\; e^{-\lambda \,t}\,\sin{\omega t}+ H(t)e^{-\lambda t}\cos{\omega t}- e^{-\lambda t}\big(\mathcal C+\frac{\lambda\, \mathcal S}{\omega}\big)\cos{\omega t}- \notag\\
& - e^{-\lambda t}\big(\mathcal S-\frac{\lambda\, \mathcal C}{\omega}\big)\sin{\omega t}- \frac{\lambda}{\omega}\, H(t)e^{-\lambda t}\sin{\omega t}\,\Big],\notag
\end{align}
where the generalized constants: $\omega, \mathcal S, \mathcal C$ in $\widehat{\mathbb R}\setminus\mathbb R$ and the generalized functions: $e^{-\lambda t}$, $\sin{\omega t}$, $\cos{\omega t}$ ,\, $H(t)e^{-\lambda t}\sin{\omega t}$ and $H(t)e^{-\lambda t}\cos{\omega t}$ in $\widehat{\mathcal E(\mathbb R)}\setminus\mathcal D^\prime(\mathbb R)$ are defined by
\begin{align}
&\omega=q(\omega_\varphi),\quad  \mathcal S=q\big(\int_{-\infty}^0e^{\lambda_\varphi x}\sin{(\omega_\varphi x)}\, \varphi(x)\, dx\big),\, \mathcal C=q\big(\int_{-\infty}^0 e^{\lambda_\varphi x}\cos{(\omega_\varphi x)}\, \varphi(x)\, dx\big),\notag\\
& e^{-\lambda t}=s^t=q\big(e^{-\lambda_\varphi t}\big),\quad \sin{\omega t}=q(\sin{\omega_\varphi t}),\quad \cos{\omega t}=q(\cos{\omega_\varphi t}),\notag\\
& H(t)e^{-\lambda t}\sin{\omega t}=q\big(\int_{-\infty}^t e^{-\lambda_\varphi(t-x)}\sin{\omega_\varphi (t-x)}\, \varphi(x)\, dx\big) \text{ and }\notag\\
& H(t)e^{-\lambda t}\cos{\omega t}=q\big(\int_{-\infty}^t e^{-\lambda_\varphi(t-x)}\cos{\omega_\varphi (t-x)}\, \varphi(x)\, dx\big),\notag
\end{align} 
respectively. 
\end{example}
\textbf{Acknowledgement:} The author thanks the organizers of the International Conference on Generalized Functions-G2014, September 2014, Southampton, U.K., where this work was first presented. The author also thanks the anonymous referee whose remarks helped to improve the quality of the text.

\end{document}